\documentclass{article}
\usepackage{theorem,graphicx,amssymb,amsmath}
\usepackage{lineno}
\usepackage[usenames,dvipsnames]{xcolor}
\usepackage[retainorgcmds]{IEEEtrantools}
\newtheorem{ours}{Theorem}
\theorembodyfont{\slshape}

\newtheorem{theorem}{Theorem}

\newtheorem{lemma}[theorem]{Lemma}
\newtheorem{prop}[theorem]{Proposition}

\newtheorem{problem}{Problem}

\def\QED{\ensuremath{{\square}}}

\def\markatright#1{\leavevmode\unskip\nobreak\quad\hspace*{\fill}{#1}}
\newenvironment{proof}
  {\begin{trivlist}\item[\hskip\labelsep{\bf Proof.}]}
  {\markatright{\QED}\end{trivlist}}

\newcommand{\diam}{\operatorname{diam}}

\newcommand{\mindist}{\operatorname{mindist}}

\title{Point Sets with Small Integer Coordinates and no Large Convex
Polygons}

\author{
  Frank Duque \thanks{Departamento de Matem\'aticas, CINVESTAV. } \footnotemark[3]\and
 Ruy Fabila-Monroy \footnotemark[1] \thanks{ruyfabila@math.cinvestav.edu.mx} \and
 Carlos Hidalgo-Toscano \footnotemark[1] \thanks{[frduque, cmhidalgo]@math.cinvestav.mx }}

\begin{document}
\maketitle

\begin{abstract}
In 1935, Erd\H{o}s and Szekeres proved that every set of $n$ points in general
position in the plane contains the vertices of a convex polygon
of $\frac{1}{2}\log_2(n)$ vertices. In 1961, they constructed, for every positive
integer $t$, a set of $n:=2^{t-2}$ points in general position in the plane, such that
every convex polygon with vertices in this set has at most $\log_2(n)+1$ vertices. 
In this paper we show how to realize their construction in
an integer grid of size $O(n^2 \log_2(n)^3)$.
\end{abstract}

\section{Introduction}

A set of points in the plane is in \emph{general position} if no three of its points
are collinear.  Let $S$ be a set of $n$ points in general position in the plane.
A \emph{convex $k$-gon of $S$} is a convex polygon of $k$ vertices whose vertices are points of $S$.
Erd\H{o}s and Szekeres~\cite{happyend} proved that every set of $\binom{2k-4}{k-2}+1$ points in general
position in the plane contains a convex $k$-gon. Using Stirling's approximation this bound can be rephrased as follows.
\begin{theorem}\textbf{(Erd\H{o}s and Szekeres~1935)}

 Every set of $n$ points in general position in the plane has a convex $k$-gon
 of at least $\frac{1}{2}\log_2 (n)$ vertices.
\end{theorem}

More than quarter of a century afterwards, Erd\H{o}s and Szekeres~\cite{erdos_cons} provided a construction of a set of points
in general position such that every convex $k$-gon of this set has logarithmic size.
Specifically, they showed the following.
\begin{theorem}\textbf{(Erd\H{o}s and Szekeres~1961)}

 For every integer $t\ge 2$, there exists a set of $n:=2^{t-2}$ points in general position in the plane
 such that every $k$-gon of this set has at most $t-1=\log_2(n)+1$ vertices.
\end{theorem}

Some inaccuracies in the construction of Erd\H{o}s and Szekeres were corrected by Kalbfleisch and
Stanton in~\cite{kal}.
The Erd\H{o}s-Szekeres construction as described in~\cite{kal} uses integer-valued
coordinates. The size of these coordinates grows quickly with respect to $n$. 
This has led some researchers to conjecture that the Erd\H{o}s-Szekeres construction 
cannot be carried out with small integer coordinates.

As an example here are some excerpts from the book ``Research Problems in Discrete Geometry''~\cite{research_problems}
by Brass, Moser and Pach regarding the Erd\H{o}s-Szekeres construction.
\begin{quote}``The complexity of this construction is reflected by the fact that none of the numerous papers
on the Erd\H{o}s-Szekeres convex polygon problem includes a picture of the $16$-point set without a
convex hexagon.''
\end{quote}
\begin{quote}``Kalbfleisch and Stanton~\cite{kal} gave explicit coordinates for the $2^{t-2}$ points in the
Erd\H{o}s-Szekeres construction. However, even in the case of $t=6$ the coordinates are so large that
they cannot be used for a reasonable illustration.''
\end{quote}
\begin{quote}
 ``The exponential blowup of the coordinates in the above lower bound constructions may be necessary.
 It is possible that all extremal configurations belong to the class of order types that have no small
 realizations.'' 
\end{quote}
Also, in the survey~\cite{survey} on the Erd\H{o}s-Szekeres problem by Morris and Soltan we find the following.
\begin{quote}
``The size of the coordinates of the points in
the configurations given by Kalbfleisch and Stanton~\cite{kal} that meet the conjectured
upper bound on $N(n)$ grows very quickly. A step toward showing that this is
unavoidable was taken by Alon et al.~\cite{restricted}.''
\end{quote}
In this paper we prove that the Erd\H{o}s-Szekeres construction can be realized
in a rather small integer grid of size. Our main result is the following.
\begin{ours}
 The Erd\H{o}s-Szekeres construction of $n=2^{t-2}$ points can be realized in an integer grid of size $O(n^2\log_2(n)^3)$.
\end{ours}

This solves an open problem
of \cite{research_problems}, which we discuss, together with other problems, in Section~\ref{sec:restricted}.

To finish this section we mention the rich history behind the improvements on the upper bound on the Erd\H{o}s-Szekeres theorem.
Let $n(k)$ be the smallest integer such that every set of $n(k)$ points in general
position in the plane contains a convex $k$-gon. 

The upper bound given by Erd\H{o}s and Szekeres is of \[n(k) \le  \binom{2k-4}{k-2}+1.\]
It took 63 years for an improvement to be found, but in the course of one year many of them followed:
Chung and Graham~\cite{chung_graham} proved that $n(k) \le  \binom{2k-4}{k-2}$;
Kleitman and Pachter~\cite{kleitman_pachter} proved that $n(k) \le  \binom{2k-4}{k-2}-2k+7;$ 
T\'oth and Valtr~\cite{toth_valtr} improved this bound roughly by a factor of $2$, they showed that $n(k) \le \binom{2k-5}{k-2}+2.$
Eight years later in 2006, T\'oth and Valtr~\cite{upper_kgon} further improved this bound by $1$.

Very recently there has been a new set of improvements. Vlachos~\cite{vlachos} 
proved that \[n(k) \le \binom{2k-5}{k-2}-\binom{2k-8}{k-3}+\binom{2k-10}{k-7}+2.\]
This implies that \[ \limsup_{k \to \infty} \frac{n(k)}{\binom{2k-5}{k-2}}\le\frac{29}{32}.\]
Afterwards, Norin and Yuditsky~\cite{without_ind} improved this to   
\[\limsup_{k \to \infty} \frac{n(k)}{\binom{2k-5}{k-2}}\le \frac{7}{8}.\]
Mojarrad and Vlachos~\cite{vlachos_2} made a furthter improvement and showed that $n(k) \le \binom{2k-5}{k-2}-\binom{2k-8}{k-3}+2$.
Finally, in 2016, Suk~\cite{suk} made a huge improvement; he proved that \[n(k)\le 2^{k+o(k)}.\]

The best lower bound is the one given by the Erd\H{o}s-Szekeres construction.  It implies that 
\[n(k)\ge 2^{k-2}+1.\] This is conjectured to be value of $n(k)$.

This paper is organized as follows. In Section~\ref{sec:cons} we describe
in detail the Erd\H{o}s-Szekeres construction and show how to realize it
in a small integer grid.  Based on the description of Section~\ref{sec:cons},
we implemented an algorithm to compute the Erd\H{o}s-Szekeres construction.
In Section~\ref{sec:imp} we discuss optimizations made on our implementation
to further reduce the size of the integer coordinates; we also provide
figures  of the Erd\H{o}s-Szekeres construction for $t=6$ and $t=7$. 
In Section~\ref{sec:restricted} we finalize the paper by
proposing Erd\H{o}s-Szekeres type
problems for point sets with integer coordinates.

\section{The Erd\H{o}s-Szekeres Construction}\label{sec:cons}
The Erd\H{o}s-Szekeres construction is made from  smaller point sets which we now describe.

\subsection{Cups and Caps}
A \emph{$k$-cup} of $S$ is a convex $k$-gon of $S$ bounded from above by a single edge;
similarly a \emph{$k$-cap} of $S$ is a convex $k$-gon of $S$ bounded from below by a single edge;
see Figure~\ref{fig:cupscaps}. Let $X$ and $Y$ be two sets of points in the plane.
We say that $X$ is \emph{high above} $Y$ if: every line determined by two points in $X$ is above every point in $Y$, and
every line determined by two points in $Y$ is below every point in $X$.

The building blocks of the Erd\H{o}s-Szekeres construction
are point sets $S_{k,l}$; which are constructed recursively as follows.

\begin{itemize} 
 \item $S_{k,l}:=\{(0,0)\}$ if $k \le 2$ or $l \le 2$; 
 \item $S_{k,l}:=L_{k,l} \cup R_{k,l}$; 

where:
\begin{IEEEeqnarray}{rCl}
L_{k,l} & :=& S_{k-1,l}; \nonumber \\
R_{k,l} &:= &\{(x+\delta_{k,l},y+\delta_{k,l}'):(x,y) \in S_{k,l-1}\}; \nonumber 
\end{IEEEeqnarray}
$\delta_{k,l}$ is chosen large enough so that $R_{k,l}$ is to the right of $L_{k,l}$; 

and
$\delta_{k,l}'$ is chosen large enough with respect to $\delta_{k,l}$ so that $R_{k,l}$ is high above $L_{k,l}$.
\end{itemize}
It can be shown by induction on $k+l$ that $S_{k,l}$ has $\binom{k+l-4}{k-2}$ points, and that $S_{k,l}$
does not contain a $k$-cup nor a $l$-cap.

\begin{figure}
  \begin{center}
   \includegraphics[width=0.7\textwidth]{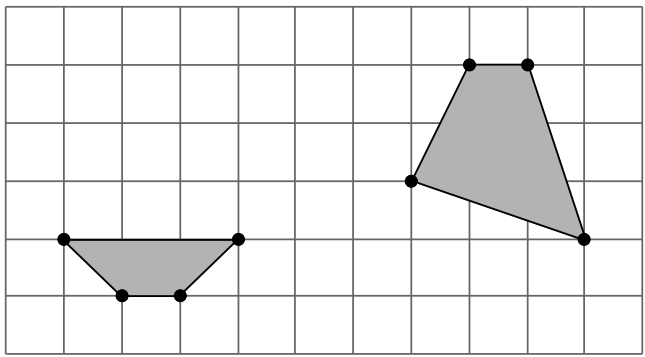}
\end{center}
\caption{A $4$-cup and a $4$-cap}
\label{fig:cupscaps}
\end{figure}

In Matou\v{s}ek's book~\cite{lectures} it is left as an exercise to show that $S_{k,l}$ can be realized
in an integer grid of polynomial size. According to Matou\v{s}ek this was noted by Valtr. We follow a different
approach; we prove that a superset of $S_{k,l}$ can be realized with small positive integer coordinates; even though
for some values of $k$ and $l$ this implies a realization of $S_{k,l}$ with integer coordinates of exponential size.

\subsection{A Superset of $S_{k,l}$}\label{sec:basic}

Let $r \ge 1$ be an integer. In this section we construct a point set $P_r$ in general position in the plane with
integer coordinates such that $P_r$ has $2^r$ points, and for $r=k+l-1$, $P_r$ contains $S_{k,l}$ as a subset.\footnote{More accurately, $P_{r}$ contains a subset with
the same order type as $S_{k,l}$.}

We define $P_r$ recursively as follows.
\begin{itemize} 
 \item $P_0:=\{(0,0)\}$; 
 \item $P_r:=L_{r} \cup R_{r}$; 

where:
\begin{IEEEeqnarray}{rCl}
L_{r} & :=& P_{r-1} \nonumber \\
R_{r} &:= &\{(x+\delta_{r},y+\delta_{r}'):(x,y) \in L_r\} \nonumber \\
\delta_r  & := & 3 \cdot 4^{r-1}; \label{eq:delta_r} \\[5pt]
\delta_r' & := & (3r+1) \cdot 4^{r-1}. \label{eq:delta_r'}
\end{IEEEeqnarray}
\end{itemize}

Let $X_r$ be the value of the largest $x$-coordinate of $P_r$;  note that for $r\ge 1$
\[X_r = X_{r-1}+\delta_r.\]
Since $X_0=0$, by induction we have that
\begin{equation}\label{eq:X}
 X_r=4^{r}-1.
\end{equation}
Let $Y_r$ be the value of the largest $y$-coordinate of $P_r$; note that for $r \ge 1$
\[Y_r=Y_{r-1}+\delta_r'.\]
Since $Y_1=0$, by induction we have that
\begin{equation}\label{eq:Y}
 Y_r= r\cdot 4^{r}.
\end{equation}
Since \[\delta_r>X_{r-1},\] every point
 of $L_r$ is to the left of every point of $R_r$.

 For $r\ge 2$, let $p_r$ be the rightmost point of $L_r$ and  let $q_r$ be the leftmost point of $R_r$;
 let $\ell_r$ be the straight line passing through $p_r$ and $q_r$; see Figure~\ref{fig:P}. 
 By construction of $P_r$, the point $p_r$ is the point of $L_r$ of largest
$y$-coordinate, and the point $q_r$ is the point of $R_r$ of smallest $y$-coordinate;  therefore, the slope $m_r$ of $\ell_r$
is given by 
\begin{IEEEeqnarray}{rCl}
m_r & = & \frac{Y_r-2Y_{r-1}}{X_r-2X_{r-1}}\nonumber \vspace{1em} \\
    & = & \frac{r4^r-2(r-1)4^{r-1}}{4^{r}-1-2(4^{r-1}-1)}\nonumber  \vspace{1em} \\
    & = & \frac{2r \cdot 4^{r-1}+2\cdot 4^{r-1}}{2\cdot 4^{r-1}+1}\nonumber \vspace{1em}\\
    & = & r+1-\frac{r+1}{2\cdot 4^{r-1}+1} \label{eq:m}. 
\end{IEEEeqnarray}

\begin{figure}
  \begin{center}
   \includegraphics[width=0.85\textwidth]{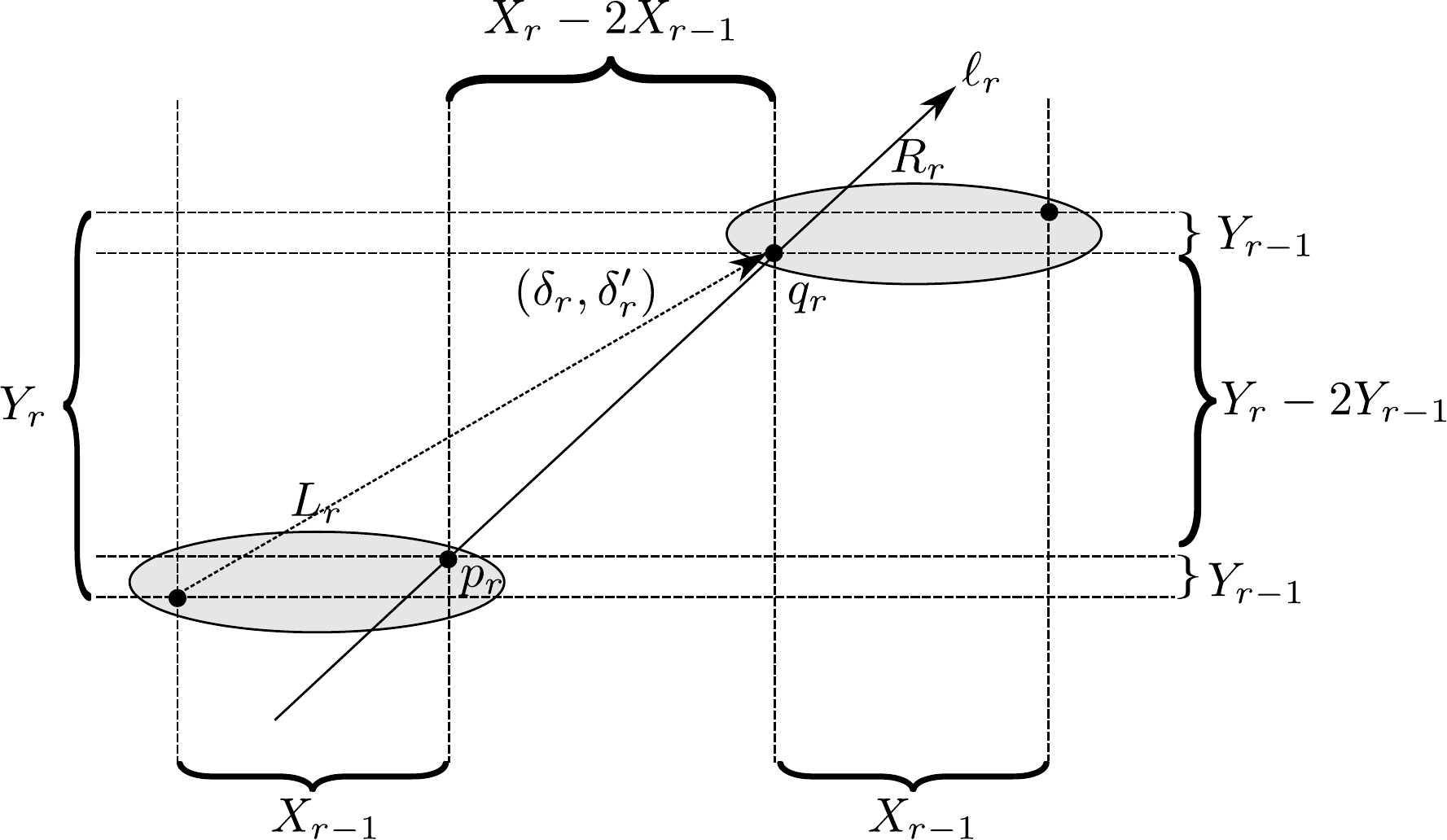}
\end{center}
\caption{$P_r$ }
\label{fig:P}
\end{figure}

The $m_r$ are increasing, since
\[ m_r-m_{r-1}  =  1+ \frac{r}{2\cdot 4^{r-2}+1}-\frac{r+1}{2 \cdot 4^{r-1}+1}>0;\]
the last inequality follows from the fact that $\frac{r+1}{2 \cdot 4^{r-1}+1}\le \frac{1}{3}$ for $r\ge 2$.

Let $L_{r-1}'$ and $R_{r-1}'$ be the translations of $L_{r-1}$ 
and $R_{r-1}$ in $R_{r}$, respectively. Let $\ell_{r-1}'$
be the line defined by the rightmost point of $L_{r-1}'$ and the leftmost
point of  $R_{r-1}'$. Thus, $\ell_{r-1}'$ is the translation of $\ell_{r-1}$ in $R_r$.
We now prove some properties of $P_r$.

\begin{lemma}\label{lem:slopes}
Among the lines passing through two points of $P_r$, $\ell_r$ is the line with the largest slope.
\end{lemma}
\begin{proof}
We proceed by induction on $r$. 
 For $r=0$ and $r=1$, the lemma holds trivially. So assume that $r>1$ and that
the lemma holds for smaller values of $r$.  Let $\ell$ be a line passing through two points of $P_r$. 

Suppose that $\ell$ passes through two points
of $L_r$ or through two points of $R_r$. By induction
the slope of $\ell$ is at most $m_{r-1}$. Since $m_r> m_{r-1}$,
the slope of $\ell_r$ is larger than the slope of $\ell$.

Suppose that $\ell$ passes through a point $p$ of $L_r$ and a point
$q$ of $R_r$. Consider the polygonal chain $C:=(p, p_r, q_r, q)$.
Since $p$ is to the left of $p_r$ and $q_r$ is to the left of $q$, the slope of $\ell$
is at most the maximum of the slopes of the edges of $C$. By induction each of these
edges has slope at most $m_{r}$. Therefore, the slope of $\ell$ is at most the slope of $\ell_r$.
\end{proof}

\begin{lemma}\label{lem:rightmost}
 The rightmost point of $P_r$ is above $\ell_{r-1}$ and the leftmost point of $P_r$ is below $\ell_{r-1}'$.
\end{lemma}
\begin{proof}
 The result holds trivially for $r=0$ and $r=1$; assume that $r\ge 2$.
  
  First we prove that the rightmost point $p$ of $P_r$ is above $\ell_{r-1}$.
  Note that $p=(X_r,Y_r)$.
 Let $q$ be the point in $\ell_{r-1}$ with $x$-coordinate equal to $X_r$;
 note that since $\ell_{r-1}$ contains the point $(X_{r-2},Y_{r-2})$,
 the $y$-coordinate of $q$ is equal to $Y_{r-2}+m_{r-1}(X_r-X_{r-2})$.
 Therefore, it is sufficient to show that: \[Y_r > Y_{r-2}+m_{r-1}(X_r-X_{r-2}). \]
 Equivalently that \[\frac{Y_r -Y_{r-2}}{X_r-X_{r-2}}>m_{r-1}.\]
 This follows from 
 \begin{IEEEeqnarray}{rCl}
  \frac{Y_r-Y_{r-2}}{X_r-X_{r-2}} & =& \frac{r4^{r}-(r-2)4^{r-2}}{4^{r}-1-(4^{r-2}-1)}\nonumber \vspace{1em} \\
                                  & = & \frac{15r\cdot 4^{r-2}+2\cdot 4^{r-2}}{15\cdot 4^{r-2}}\nonumber \vspace{1em}\\
                                  & = & r+\frac{2}{15},\nonumber
 \end{IEEEeqnarray}
and that by (\ref{eq:m})
 \[m_{r-1}=r-\frac{r}{2\cdot 4^{r-2}+1}.\]
 
 Now we prove that the leftmost point of $P_r$ is below $\ell_{r-1}'$. Note that $(0,0)$ is the leftmost point of $P_r$. 
 Let $q'$ be the point in $\ell_{r-1}'$ with $x$-coordinate equal to $0$;
 note that since $\ell_{r-1}'$ contains the point $(X_{r-2}+\delta_r,Y_{r-2}+\delta_r')$,
 the $y$-coordinate of $q'$ is equal to $Y_{r-2}+\delta_r'-m_{r-1}(X_{r-2}+\delta_r)$.
 Therefore, it is sufficient to show that: \[Y_{r-2}+\delta_r'-m_{r-1}(X_{r-2}+\delta_r)>0. \]
 Equivalently that \[\frac{Y_{r-2}+\delta_r'}{X_{r-2}+\delta_r} > m_{r-1}.\]
 This follows from 
 \begin{IEEEeqnarray}{rCl}
  \frac{Y_{r-2}+\delta_r'}{X_{r-2}+\delta_r} & =& \frac{(r-2)4^{r-2}+(3r+1)4^{r-1}}{4^{r-2}-1+3 \cdot 4^{r-1}}\nonumber \vspace{1em} \\
                                  & = & \frac{13r\cdot 4^{r-2}+2\cdot 4^{r-2}}{13\cdot 4^{r-2}-1}\nonumber \vspace{1em}\\
                                  & = & r+\frac{2 \cdot 4^{r-2}+r}{13 \cdot 4^{r-2}-1},\nonumber
 \end{IEEEeqnarray}
and that by (\ref{eq:m})
 \[m_{r-1}=r-\frac{r}{2\cdot 4^{r-2}+1}.\]
\end{proof}

\begin{lemma}\label{lem:ell_prop}
  The following properties hold.
 
 \begin{itemize}
 \item[(a)] $R_r$ is above $\ell_{r-1}$;
 \item[(b)] $L_r$ is below $\ell_{r-1}'$;
 \item[(c)] no point of $L_{r-1}$ is below $\ell_{r-1}$;
 \item[(d)] no point of $L_{r-1}'$ is below $\ell_{r-1}'$; 
 \item[(e)] no point of $R_{r-1}$ is above $\ell_{r-1}$; and
 \item[(f)] no point of $R_{r-1}'$ is above $\ell_{r-1}'$.
\end{itemize}
\end{lemma}
\begin{proof}
 For $r=0,1,2$ the lemma can be verified directly or holds trivially; assume that $r>2$.
 
 \begin{itemize}
  \item[\emph{(a)}]
  By Lemma~\ref{lem:rightmost}, the rightmost point of $R_{r}$ is above $\ell_{r-1}$.
  If a point $p$ of $R_r$ is below $\ell_{r-1}$, then the 
line defined by $p$ and the rightmost point of $R_r$ has slope larger than $m_{r-1}$---
a contradiction to Lemma~\ref{lem:slopes} and the fact that
$R_r$ is a translation of $P_{r-1}$. 

\item[\emph{(b)}]
  By Lemma~\ref{lem:rightmost}, the leftmost point of $L_{r}$ is below $\ell_{r-1}'$.
  If a point $p$ of $L_r$ is above $\ell_{r-1}'$, then the 
line defined by $p$ and the leftmost point of $L_r$ has slope larger than the slope of $\ell_{r-1}'$---
a contradiction to Lemma~\ref{lem:slopes} and the fact that
$\ell_{r-1}'$ is parallel to $\ell_{r-1}$. 

\item[\emph{(c)}]
If a point $p$ of $L_{r-1}$ is below $\ell_{r-1}$, then the line defined by
$p$ and the rightmost point of $L_{r-1}$ has slope larger than $m_{r-1}$---
a contradiction to Lemma~\ref{lem:slopes}. 

\item[\emph{(d)}]
Follows from \emph{(c)} and the fact that $R_r$ is a translation of $L_r$.

\item[\emph{(e)}]
If a point $p$ of $R_{r-1}$ is above $\ell_{r-1}$, then the line defined by
$p$ and the leftmost point of $R_{r-1}$ has slope larger than $m_{r-1}$---
a contradiction to Lemma~\ref{lem:slopes}. 

\item[\emph{(f)}]
Follows from \emph{(e)} and the fact that $R_r$ is a translation of $L_r$.
 \end{itemize}

\end{proof}

\begin{lemma}\label{lem:high_above}
 $R_r$ is high above $L_r$.
\end{lemma}
\begin{proof}  
For $r=0,1,2$ the lemma holds trivially or can be verified directly; assume that $r>2$
 and that lemma holds for smaller values of $r$. We proceed by induction
 on $r$.
 
We first prove that $R_r$ is above every line $\ell$ defined by two points of 
$L_r$. By $\emph{(a)}$ of Lemma~\ref{lem:ell_prop}, $R_{r}$ is above $\ell_{r-1}$. 
By Lemma~\ref{lem:slopes}, the slope of $\ell$ is at most the slope of $\ell_{r-1}$.
Thus we may assume that $\ell$ does not contain the rightmost point
of $L_{r-1}$ nor the leftmost point of $R_{r-1}$. 
Suppose that $\ell$ passes through a point of $R_{r-1}$; 
then, by \emph{(e)} of Lemma~\ref{lem:ell_prop}, this point is below $\ell_{r-1}$.
 Since the slope of $\ell$
 is at most the slope of $\ell_{r-1}$, $R_r$ is above $\ell$ in this case.
Suppose that $\ell$ passes through two points of $L_{r-1}$. Then, by induction
the leftmost point of $R_{r-1}$ is above $\ell$. 
Since the slope of $\ell$
is at most the slope of $\ell_{r-1}$, $R_r$ is above $\ell$ in this case.
 
 We now prove that $L_r$ is below every line $\ell$ defined by two points of $R_r$.
By \emph{(b)} of Lemma~\ref{lem:ell_prop}, $L_{r}$ is below $\ell_{r-1}'$. Since
$R_r$ is a translation of $P_{r-1}$, by Lemma~\ref{lem:slopes}, we have that the slope of $\ell$ is at most the slope of $\ell_{r-1}'$.
Thus we may assume that $\ell$ does not contain the rightmost point
of $L_{r-1}'$ nor the leftmost point of $R_{r-1}'$. 
Suppose that $\ell$ passes through a point of $L_{r-1}'$; 
then, by \emph{(d)} of Lemma~\ref{lem:ell_prop}, this point is above $\ell_{r-1}'$.
 Since the slope of $\ell$ is at most the slope of $\ell_{r-1}$, $L_r$ is below $\ell$ in this case.
Suppose that $\ell$ passes through two points of $R_{r-1}'$. Since $R_r$ is translation
of $P_{r-1}$, by induction we have that 
the rightmost point of $L_{r-1}'$ is below $\ell$. 
Since the slope of $\ell$
is at most the slope of $\ell_{r-1}'$, $R_r$ is above $\ell$ in this case.
The result follows.
\end{proof}

\begin{prop}\label{prop:grid}
 $P_r$ can be realized with non-negative integer coordinates of size at most
 $r4^r$.
\end{prop}
\begin{proof}
 This follows from 
  $X_r=4^r-1$ and $Y_r=r4^r$.
\end{proof}

\begin{prop}\label{prop:subset}
 Let $k,l$ be positive integers. If $r:=k+l-1$ then $S_{k,l}$ is a subset of $P_r$.
\end{prop}
\begin{proof}
 The result holds for $k\le 2$ or $l \le 2$, since in these cases
 $S_{k,l}=\{(0,0)\}$. Therefore, the result holds for $r\le 4$.
 Assume that $k,l \ge 2$, $r \ge 5$ and that the result holds for smaller
 values of $r$.
 By induction $S_{k-1,l}$ and $S_{k,l-1}$ are subsets of $P_{r-1}$. The result
 follows from Lemma~\ref{lem:high_above} and by setting $\delta_{k,l}:=\delta_{r}$
 and $\delta_{k,l}':=\delta_{r}'$.
\end{proof}

\subsection{The Erd\H{o}s-Szekeres Construction with Small Integer Coordinates}\label{sec:cons_small}

In this section we use the set of points described in Section~\ref{sec:basic}
to realize, with small integer coordinates, the construction given by Erd\H{o}s and Szekeres in \cite{erdos_cons}.
The Erd\H{o}s-Szekeres construction is made from a small number of translations of $S_{k,l}$ 
(for the pairs $(k, l)$, where $k+l$ is fixed). We first describe these translations.

Let $t > 0$ be an integer and let $n:=2^{t-2}$. For every integer $1 \le i \le t-2$ we  define the vector
\[v_i:=(3(t-i),-3i).\]
Using these vectors, we define a set of $t-1$ points $w_0,\dots, w_{t-2}$ recursively as follows.
\begin{itemize}
 \item $w_0:=(0,0)$;
 
 \item $w_{i+1}:=w_{i}+v_{i}$ for $i=0,\dots, t-3$.
\end{itemize}
For $i=0,\dots,t-2$, let $C_i$ be the unit square whose lower left corner is equal to $w_i$.

\begin{lemma}\label{lem:union_squares}
 The union, $\bigcup C_i$, of the squares $C_i$ lies in a $3t^2 \times 3t^2$ integer grid.
\end{lemma}
\begin{proof}
Note that the largest absolute value of the $x$-coordinates of the $w_i$'s is equal to 
\[\sum_{i=0}^{t-3} (3t-i)< 3t^2;\]
and the  largest absolute value of the $y$-coordinates of the $w_i$'s is equal to 
\[\sum_{i=0}^{t-3} 3i< 3t^2.\]
Therefore, $\bigcup C_i$ lies in a $3t^2 \times 3t^2$ integer grid.
\end{proof}

Let $D_i$ be the square $C_i$ scaled by factor of $(t+1)4^{t+1}$,
that is \[D_{i}:=\{\left ( (t+1)4^{t+1}x ,(t+1)4^{t+1} y \right ):(x,y) \in C_i \}.\]

\begin{lemma}\label{lem:turn}
 Let $0 \le i < j < k \le r-2$ be three integers; let $p_i, p_j,$  and $p_k$ be three points
in $D_i, D_j$ and $D_k$, respectively. Then $(p_i, p_j, p_k)$ is a right turn.
\end{lemma}
\begin{proof}
 For $i=0,\dots,t-3$, let $W_i$ be the set of vectors 
 of the form $u:=q-q'$ where $q$ is a point of $C_{i+1}$ and $q'$ is a point
 of $C_{i}$. Note that the endpoints of these vectors lie in a $2\times 2$ square centered at $v_i$; let $\gamma_i$ be
the smallest cone with apex at the origin and that contains $C_i$. By the previous observation the $\gamma_i$
only intersect at the origin; see Figure~\ref{fig:vw}.

\begin{figure}
  \begin{center}
   \includegraphics[width=1.0\textwidth]{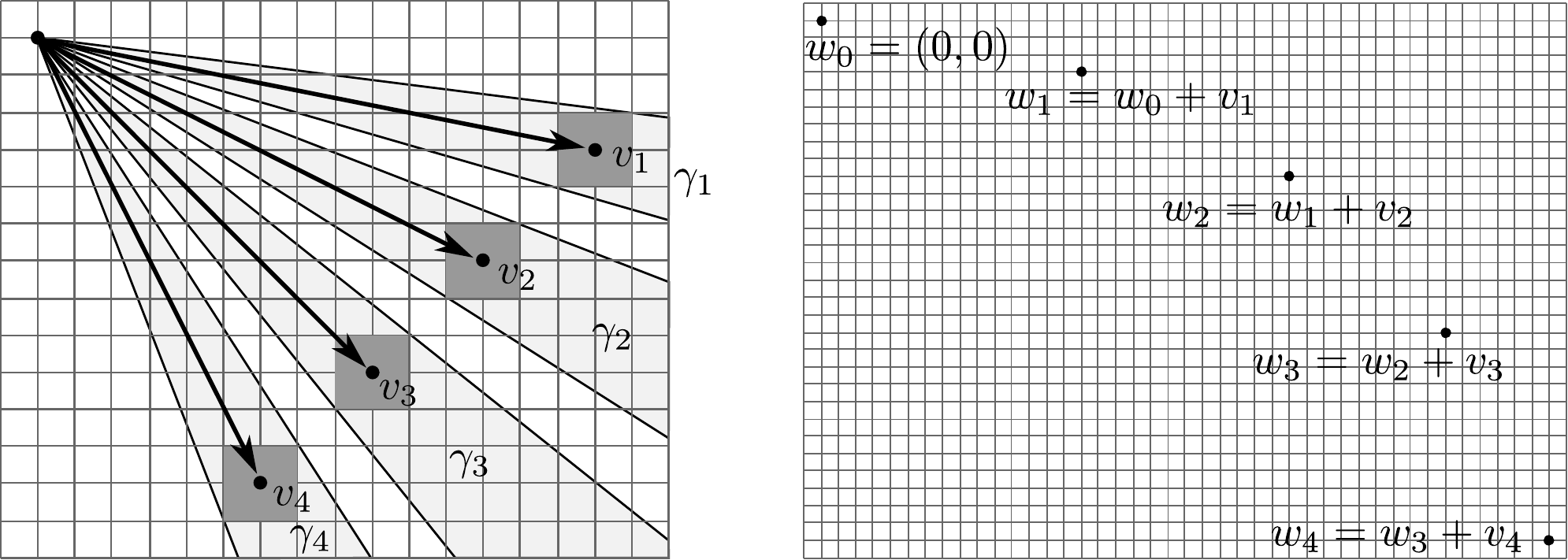}
\end{center}
\caption{The vectors $v_i$, the cones $\gamma_i$ and the points $w_i$, for $t=6$}
\label{fig:vw}
\end{figure}

Let $m_{i-1,i}$ be the slope of a line passing through a point of $C_{i-1}$ and a point of $C_{i}$
and let $m_{i,i+1}$ be the slope of a line passing through a point of $C_{i}$ and a point of $C_{i+1}$.
The vector defining $m_{i-1,i}$ lies in $\gamma_{i-1}$ and the 
 vector defining $m_{i,i+1}$ lies in $\gamma_{i}$.
This implies that $m_{i-1,i} >  m_{i,i+1}$. Let
$0 \le i < j < k \le r-2$ be three integers. Let $m_{i,j}$ be the slope of a line passing through a point of $C_{i}$ and a point of $C_{j}$
and let $m_{j,k}$ be the slope of a line passing through a point of $C_{j}$ and a point of $C_{k}$.
Thus, we have that
\begin{equation}\label{eq:m_i}
m_{i,j} > m_{j,k}.
\end{equation}
Note that (\ref{eq:m_i}) also holds for the lines defined by pairs points in the $D_i$'s.
Therefore, $(p_i, p_j, p_k)$ is a right turn.
\end{proof}

Let $q_i$ be the lower left corner of $D_i$, and let $S_{t-i,i+2}'$ be the translation of $S_{t-i,i+2}$
by $q_i$. That is \[S_{t-i,i+2}':=\{p+q_i:p \in S_{t-i,i+2} \}.\]
The Erd\H{o}s-Szekeres construction is given by \[S_t = \bigcup_{i=0}^{t-2} S_{t-i,i+2}'.\]
Note that \[|S_t|=\sum_{i=0}^{t-2} |S_{t-i,i+2}|=\sum_{j=0}^{t-2} \binom{t-2}{j}=2^{t-2}=n.\]

\begin{prop}
 $S_t$ lies in an integer grid of size in an integer grid of size $O(n^2\log_2(n)^3)$.
\end{prop}
\begin{proof}
 Recall that $D_i$ is a scaling of $C_i$ by a factor of $(t+1)4^{t+1}$.
 Therefore, by Lemma~\ref{lem:union_squares}, $S_t$ lies in an integer
 grid of size   \[ 3t^2(t+1) 4^{t+1} = 192n^2\log_2(4n)^2 \log_2(8n)=O(n^2 \log_2(n)^3).\] 
\end{proof}

\begin{prop}
 $S_t$ is in general position.
\end{prop}
\begin{proof}
 Let $p_1, p_2, p_3$  be three points of $S_t$. If these three points
 are contained in a same $D_i$, then they are not collinear since $S_{k,l}$ is in general
 position. If the three of them are in different $D_i$, then by Lemma~\ref{lem:turn}  they are not collinear. If two of them lie on a same $D_i$ and one of them
 in some $D_j$, then they are not collinear since the slope of a line joining a point in $D_i$ and a 
 point in $D_j$ is negative, while the slope of a line defined by two points in $S_{k,l}$
 is greater or equal to zero. Therefore, $S_t$ does not contain three collinear points.
\end{proof}

\begin{theorem}
 Every convex $k$-gon of $S_t$ has at most $t-1$ vertices.
\end{theorem}
\begin{proof}
 Let $P$ be a convex $k$-gon of $S_t$. Let $U$ and $L$ be the upper and lower
 polygonal chains of $P$, respectively. Let $s$ be the index such that the 
 leftmost point of $U$ (and $L$) is in $D_s$, and let $r$ be the index
 such that the rightmost point of $U$ (and $L$) is in $D_r$.
 
 Note that for all $0 \le i < j \le t-2$,
 the slopes of an edge joining a point of $D_{i}$ with a point of $D_{j}$ are negative;
 since the slope of an edge defined by a pair of points in  $S_{t-i,i+2}'$ is
 greater or equal to zero, neither $U$ nor $L$ contain two consecutive vertices in $D_i$ for
 $s < i <r$. By Lemma~\ref{lem:turn} such a $D_i$ cannot contain a vertex of
 $L$. Therefore, $P$ contains at most $r-s-1$ vertices not in $D_s \cup D_r$.
 
 The vertices of $P$ contained in $S_s$ must form a cap and thus consists of at most $s+1$ vertices. 
 Similarly, the vertices of $P$ contained in $S_r$ must form a cup and therefore consists of at most 
 $t-r-1$ vertices. Therefore, $P$ has at most $(r-s-1)+(s+1)+(t-r-1)=t-1$ vertices; the result follows.
\end{proof}

To summarize, we have the following. 
\begin{theorem}\label{thm:es_size}
 The Erd\H{o}s-Szekeres construction of $n=2^{t-2}$ points can be realized in an integer grid of size $O(n^2\log_2(n)^3)$.
\end{theorem}

\section{Implementation}\label{sec:imp}

A direct implementation of  Section~\ref{sec:cons_small}
gives way to an efficient algorithm to compute the Erd\H{o}s-Szekeres construction.
By Proposition~\ref{thm:es_size}, the size of the grid needed by this algorithm is asymptotically
small; however, there are large constants hidden in such an implementation.
 In this section we mention some optimizations we have done to further
reduce the size of the integer grid needed for the Erd\H{o}s-Szekeres construction.

\begin{itemize}

\item \textbf{Decrease the horizontal distance between left and right parts of $S_{k,l}$.}

In Section~\ref{sec:basic}, to construct $P_r$, we gave explicit values to $\delta_r$ and $\delta_r'$. This allowed us
to show that $L_r$ is to the left of $R_r$ and that $R_r$ is high above $L_r$. However, it is 
enough to show that $L_r$ is to the left of $R_r$ and that the rightmost point of $R_r$ is above
$\ell_r$. 
That is that 
\begin{equation}\label{eq:rec}
 Y_r > \frac{Y_r-Y_{r-2}}{X_r-X_{r-2}} \left( X_r-X_{r-2} \right)+Y_{r-2}.
\end{equation}
Let $c>0$ be a constant and set $X_r=(2+c)^r$. It can be shown
that if we replace inequality (\ref{eq:rec}) by an equality and solve
for $Y_r$, then $Y_r$ is of order $O\left ( \left (2+\frac{3}{c}\right )^r \right )$.
Therefore, if we set $c=\sqrt{3}$, then both $X_r$ and $Y_r$ are of order 
$O \left  ( \left ( 2+\sqrt{3} \right )^r \right)$. In the actual implementation we
choose $\delta_r$ so that \[X_r = \left \lceil \left (2+\sqrt{3} \right )^r \right \rceil. \]
Then we choose $\delta_r'$ so that \[Y_r =\left \lceil \frac{Y_r-Y_{r-2}}{X_r-X_{r-2}} \left( X_r-X_{r-2} \right)+Y_{r-2}+1\right \rceil.\]
The addition of the ceiling functions has prevented us from proving that $Y_r$ is of order
$O\left ( \left (2+\frac{3}{c}\right )^r \right )$. If this is the case then $P_r$ can be realized 
in an integer grid of size $O\left (n^{\log_2(2+\sqrt{3})}\right )=O \left (n^{1.8999\dots} \right )$.
In Section~\ref{sec:basic}, we opted to avoid
using ceiling functions at the expense of being able to show a slightly worse upper bound.

Inspired by this, we do likewise when constructing $S_{k,l}$. First we construct $S_{k-1,l}$ and  $S_{k,l-1}$.
Let $X_{k,l}$ be the horizontal length of $S_{k,l}$. 
We choose \[ \delta_{k,l} := \left \lceil  \left (1+\sqrt{3} \right ) \left ( \frac{X_{k-1,l}+X_{k,l-1}}{2} \right )\right \rceil .\]

For any two positive integers $k$ and $l$, let $\ell_{k,l}$ be the straight line passing through the rightmost point of $S_{k-1,l}$
and the leftmost point of the copy of $S_{k,l-1}$ in $S_{k,l}$. (This definition is similar to the definition of $\ell_r$ for $P_r$.)
We choose $\delta_{k,l}'$ so that; the rightmost point in the translation of $S_{k,l-1}$ is above
$\ell_{k-1,l}$; and the leftmost point of $S_{k-1,l}$ is below the corresponding
translation of $\ell_{k,l-1}$ in $R_{k,l}$.

\item \textbf{Separate the left and right parts of $S_{k,l}$ by one in the last step of the recursion.}

The reason for choosing a relatively large horizontal separation between the left and right 
parts of $S_{k,l}$ is so that the slope of $\ell_{k,}$ does not increase too quickly. We do not need
to do this in the last step of the construction.
At each step in the construction of $S_{i,j}$ for $2 \le i \le k$, $2 \le j \le k$ and $i+j<k$,
we separate the corresponding left and right parts as before. In the last step, when constructing
$S_{k,l}$, we separate $S_{k-1,l}$ from the copy of $S_{k,l-1}$ by one.

\item \textbf{Decrease the size of the squares (rectangles) $D_i$.}

In Section~\ref{sec:cons_small}, for $i=0,\dots,t-2$ we defined a square $D_i$,
inside which we placed a copy of $S_{t-i,i+2}$. $D_i$ was chosen large
enough so that $P_r$ fits inside $D_i$ for $r=t+1$. Since we only need to fit $S_{t-i,i+2}$
we replace $D_i$ by a rectangle of length $X_{t-1,i+2}$ and height $Y_{t-i,i+2}$. 
The definitions of the $v_i$'s and $w_i$'s are changed accordingly.
\end{itemize}

In Figure~\ref{fig:S55} we show $S_{5,5}$ in $55 \times 109$ integer grid; in Figure~\ref{fig:S_6}
we show $S_{6}$ in a $58 \times 62$ integer grid; and in Figure~\ref{fig:S_7} we show
$S_{7}$ in a $230 \times 310$ integer grid. For a comparison, we note that 
Kalbfleisch and Stanton~\cite{kal} realize $S_6$
in a $6970 \times 1828$ integer grid.

Our implementations are freely available as part of \emph{``Python's Discrete and Combinatorial Geometry Library''}
at \texttt{www.pydcg.org}. They are included in the \emph{points} module. In that module the following
constructions are also included: Convex Position and the Double Circle as described in \cite{double_circle};
the Horton Set as described in~\cite{horton_us} and the Squared Horton set (this set was first defined by Valtr
in~\cite{restricted_valtr}).

\begin{figure}
  \begin{center}
   \includegraphics[width=0.8\textwidth]{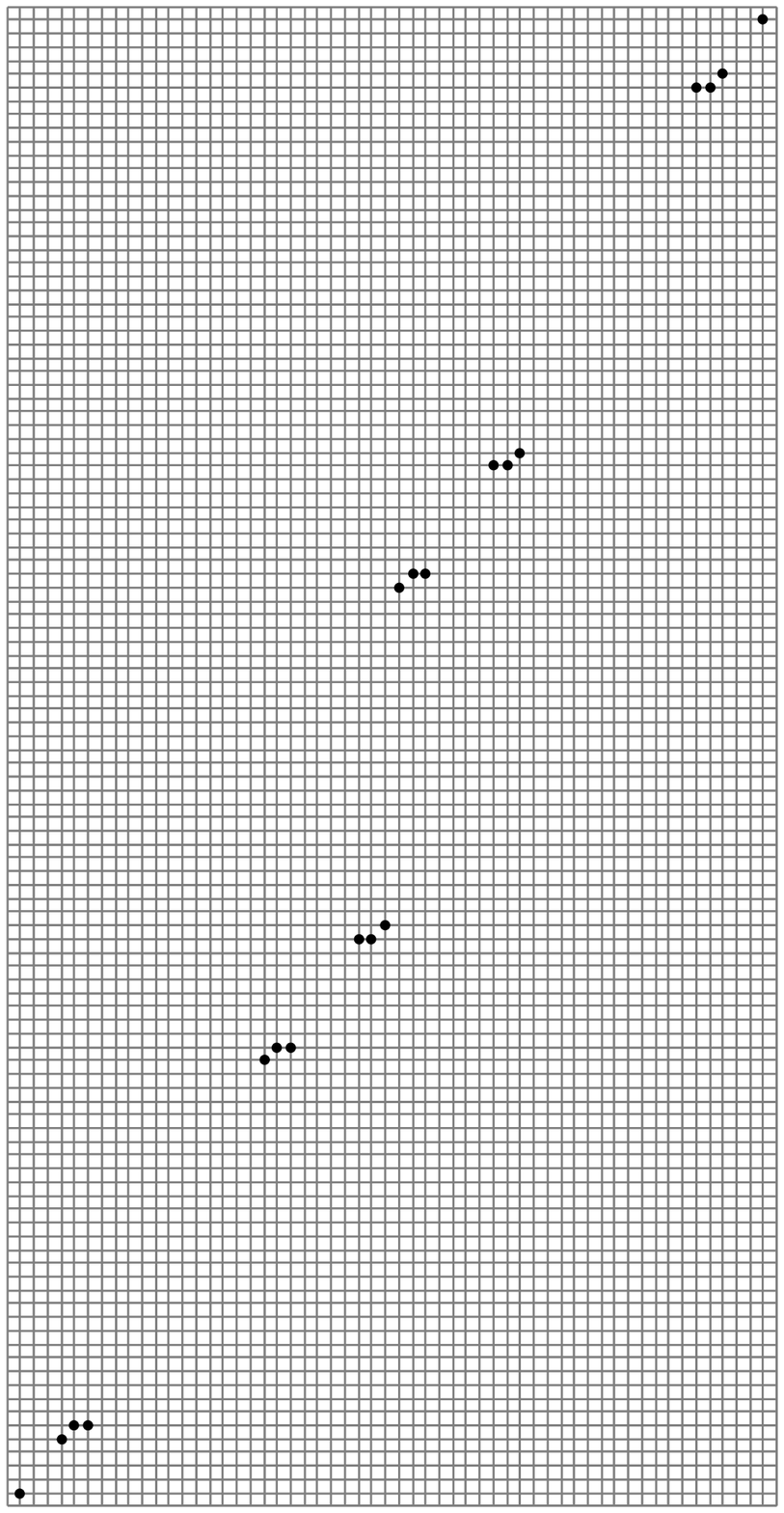}
\end{center}
\caption{$S_{5,5}$ }
\label{fig:S55}

\end{figure}
\begin{figure}
  \begin{center}
   \includegraphics[width=0.85\textwidth]{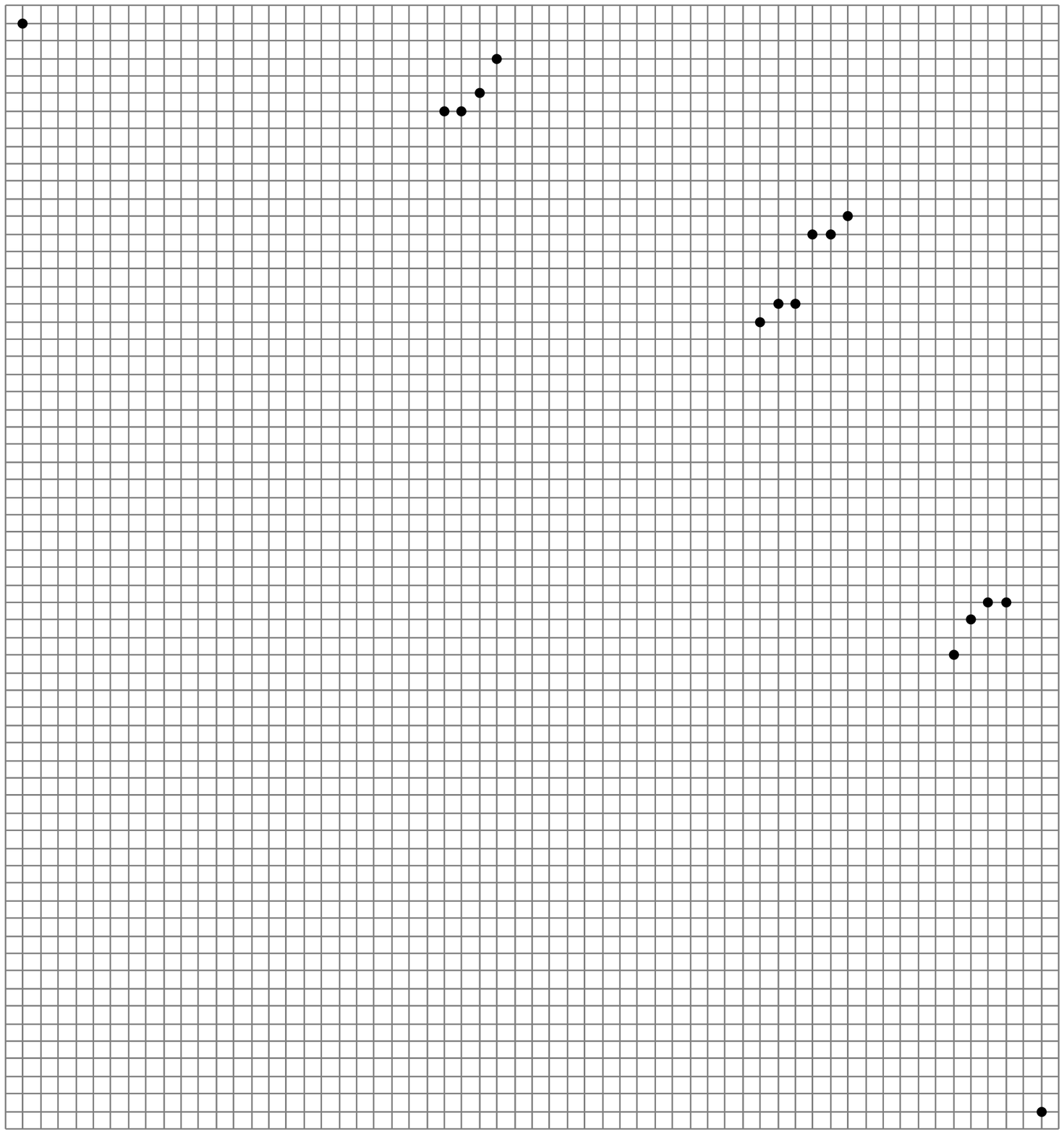}
\end{center}
\caption{$S_6$ }
\label{fig:S_6}
\end{figure}

\begin{figure}
  \begin{center}
   \includegraphics[width=1.0\textwidth]{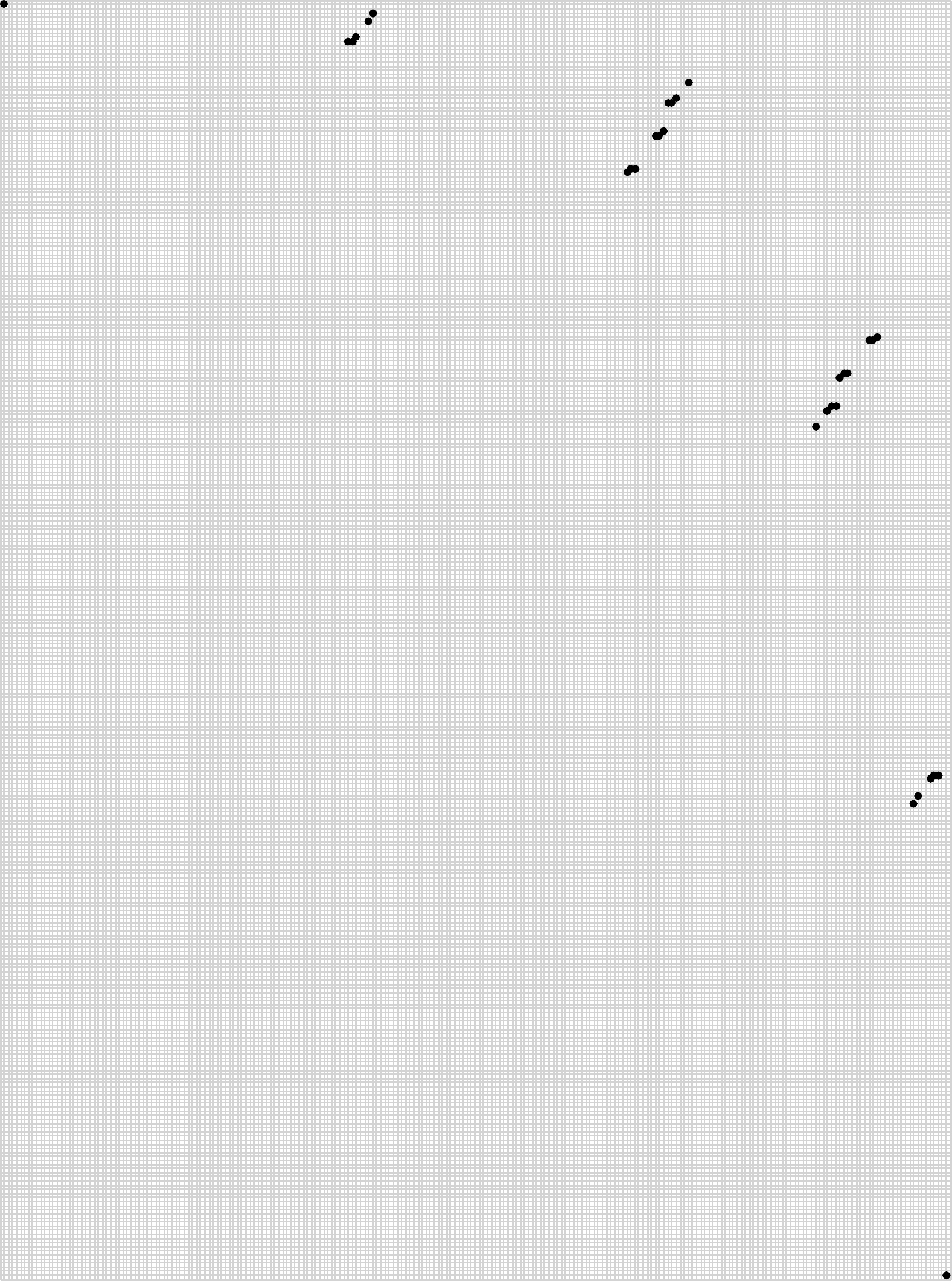}
\end{center}
\caption{$S_7$ }
\label{fig:S_7}
\end{figure}


\newpage
\section{Erd\H{o}s-Szekeres Type Problems in Restricted Planar Point Sets}\label{sec:restricted}
In this section we propose some open Erd\H{o}s-Szekeres type problems on point sets in integer grids.
Let $\diam(S)$ be the maximum distance between a pair of points of $S$, and 
let $\mindist(S)$ be the minimum distance between a pair of points of $S$.  Alon, Katchalski and  Pulleyblank~\cite{restricted},
showed that if for some constant $\alpha >0$, $S$ satisfies 
\[ \frac{\diam(S)}{\mindist(S)} \le \alpha n^{\frac{1}{2}},\]
then $S$ contains a convex $k$-gon of $\Omega \left ( n^{\frac{1}{4}} \right )$ vertices;
in ~\cite{restricted_valtr}, Valtr improved this bound to $\Omega \left ( n^{\frac{1}{3}} \right )$. He also 
showed that if for some constant $\alpha >0$, $S$ satisfies 
\[ \frac{\diam(S)}{\mindist(S)} \le \alpha \sqrt{n},\]
then $S$ contains a convex $k$-gon of $\Omega \left (n^{\frac{1}{3}} \right )$ vertices. That is, metric restrictions
on $S$ may force large convex polygons. 

This prompted the following two problems in \cite{research_problems}.
\begin{problem}\label{prob:ratio}
 Does there exist, for every $\beta \ge 1$, a suitable constant $\varepsilon(\beta)>0$ with the following 
 property: any set of $S$ of $n$ points in general position in the plane with $\frac{\diam(S)}{\mindist(S)} < n^\beta$
 contains a convex $n^{\varepsilon(\beta)}$-gon?
\end{problem}
\begin{problem}\label{prob:int}
 Does there exist, for every $\gamma \ge 1$, a suitable constant $\varepsilon(\gamma)>0$ with the following property:
 any set of $n$ points in the general position in the plane with positive integer coordinates that do not exceed
 $n^{\gamma}$ contains a convex $n^{\varepsilon(\gamma)}$-gon?
\end{problem}

Valtr in his PhD thesis~\cite{pavel_thesis} showed that the answer for Problem~\ref{prob:ratio} is \textbf{``yes''} for $\beta <1$.
He also noted, in passing, in page 55 of his thesis the following. 
\begin{quote}
 ``If $\tau=1/2$ then it is possible to construct an $\left (n^{\tau}=\sqrt{n} \right)$-dense set of size
$n$ which contains no more than $O(\log n)$ vertices of a convex polygon. Such a set can be obtained by 
an affine transformation from the construction of Erd\H{o}s and Szekeres~\cite{erdos_cons}...''
\end{quote}
$S$ is said to be \emph{$\alpha$-dense} if $\frac{\diam(S)}{\mindist(S)} \le \alpha \sqrt{n}$. So this observation
solves Problem~\ref{prob:ratio} for $\beta \ge 1$ as well. Problem~\ref{prob:int} appeared first
in Valtr's thesis (Problem~10), where it is attributed to Welzl. 

In this paper we have shown
that the answer to Problem~\ref{prob:int} is \textbf{``No''} for all $\gamma > 2$.
We conjecture that the answer  to Problem~\ref{prob:int} is \textbf{``No''}
for all $\gamma \ge 1$. However, we conjecture that there exists a $\gamma \in (1,2)$
such that the Erd\H{o}s-Szekeres construction cannot be realized in an $n^\gamma \times
n^\gamma$ integer grid. We propose the following alternative to Problem~\ref{prob:int}.

\begin{problem}
  Does there exist, for every $\gamma \ge 1$ and every $n >0$, a suitable constant $\varepsilon(\gamma)>0$ and a set $S$
  of $n$ points in general position in the plane
  with the following property?
 $S$ has positive integer coordinates not exceeding $n^\gamma$, and $S$ does not contain a convex $\varepsilon(\gamma)\log_2(n)$-gon.
\end{problem}

\subsection{Empty $k$-gons}

A convex $k$-gon of $S$ is \emph{empty} if it does not contain a point of $S$ in its interior.
In 1978, Erd\H{o}s~\cite{somemore} asked whether an analogue of the Erd\H{o}s-Szekeres theorem
holds for empty convex $k$-gons. That is, if for every $k$, every sufficiently large point set in general
position in the plane contains an empty $k$-gon. 

Every set of at least three points contains an empty triangle; Esther Klein~\cite{happyend} proved that every set of five
points contains an empty convex $4$-gon; Harborth~\cite{harborth} proved that every set
of 10 points contains an empty convex $5$-gon; and Horton~\cite{horton} constructed arbitrarily large point
sets without empty convex $7$-gons. (His construction is now known as the \emph{Horton Set}.)
The question for convex $6$-gons remained open for more than a quarter of a century
until Nicol\'as~\cite{nicolas} and independently Gerken~\cite{gerken} showed that every sufficiently large set of points
contains a convex empty $6$-gon.

Alon et al. posed a problem in~\cite{restricted}, similar to Problem~\ref{prob:ratio}, but for
empty convex $k$-gons. They asked whether if for some constant $\alpha >0$ every
sufficiently large $\alpha$-dense set of points contains an empty convex $7$-gon.
Valtr in \cite{restricted_valtr} showed that there exist arbitrarily large $\sqrt{2\sqrt{3}/\pi}$-dense point
sets not containing an empty convex $7$-gon. His construction is based on the Horton set.

The same question can be asked for point sets in an integer grid. 
\begin{problem}\label{prob:int_empty}
 Does the following hold for every constant $\gamma \ge 0$? 
 Every sufficiently large set of $n$ points in the general position in the plane with positive integer coordinates that do not exceed
 $n^{\gamma}$ contains an empty convex $7$-gon.
\end{problem}
As far as we know all constructions without an empty convex $7$-gons are based on the Horton set.
This is particularly relevant for Problem~\ref{prob:int_empty} for the following reason.
In \cite{horton_us}, Barba, Duque, Fabila-Monroy and Hidalgo-Toscano proved that the Horton set
cannot be realized in an integer grid of polynomial size.

\subsection*{Acknowledgments} 

We thank Luis Felipe Barba for providing us with a copy of \cite{kal}.

\small
\bibliographystyle{abbrv} \bibliography{small}

\end{document}